\documentclass[10pt]{amsart} 
\usepackage{latexsym}
\usepackage{amsmath}
\usepackage{amsfonts}
\usepackage{amssymb}
\usepackage{tikz,float}
\usepackage{mathrsfs}
\usepackage{cite}
\usepackage{enumitem}
\usepackage{listings}
\usepackage{blkarray}
\usepackage{subcaption}
\newtheorem{theorem}{Theorem}[section]
\newtheorem{lemma}[theorem]{Lemma}
\newtheorem{observation}[theorem]{Observation}
\newtheorem{corollary}[theorem]{Corollary}
\newtheorem{proposition}[theorem]{Proposition}

\newtheorem{sublemma}{}[theorem]

\theoremstyle{definition}

\newtheorem{example}[theorem]{Example}

\theoremstyle{remark}

\numberwithin{equation}{section}

\newcommand{\romannum}[1]{\romannumeral#1\relax}
\newcommand{\cN}{\mathcal{N}}

\usepackage[margin=1.75in]{geometry}  

\newcommand{\supp}{\text{supp}}

\long \def \ignore #1 {}

\title{Reconstruction of $C_4$-Free Graphs from the Set of Closed Neighborhoods and Digital Convexity}

\author{Steffen Borgwardt}
\address{Department of Mathematical and Statistical Sciences\\
University of Colorado Denver\\
Denver, CO}
\email{steffen.borgwardt@ucdenver.edu}

\author{MacKenzie Carr}
\address{Department of Mathematics\\ Toronto Metropolitan University \\ Toronto, ON}
\email{mackenzie.carr@torontomu.ca}

\author{Ce Chen}
\address{Mathematics Department\\
University of Illinois Urbana-Champaign \\ Urbana, IL}
\email{cechen4@illinois.edu}

\author{Wayne Ge}
\address{Mathematics Department \\
Louisiana State University \\ Baton Rouge, LA}
\email{yge4@lsu.edu}

\author{Stephen G. Hartke}
\address{Department of Mathematical and Statistical Sciences\\
University of Colorado Denver\\
Denver, CO}
\email{stephen.hartke@ucdenver.edu}

\author{Yixuan Huang}
\address{Mathematics Department \\
Vanderbilt University \\ Nashville, TN}
\email{yixuan.huang.2@vanderbilt.edu}

\author{Alex Moon}
\address{Department of Mathematics \\ Dartmouth College \\ Hanover, NH}
\email{alexander.j.moon.gr@dartmouth.edu}

\date{}

\begin{document}

\begin{abstract}
    Fomin, Kratochv\'il, Lokshtanov, Mancini, and Telle showed that every $C_{4}$-free graph is reconstructible from the \emph{multiset} of closed neighborhoods.  We strengthen their result proving that every $C_{4}$-free graph is reconstructible from the \emph{set} of closed neighborhoods. 
    
    This extends the work of Lafrance et al.\ by showing that all $C_{4}$-free graphs, and hence all graphs of girth at least five, are reconstructible from their digitally convex sets.
    A subset $S$ of vertices in a graph $G$ is digitally convex if, for every vertex $v \notin S$, there is a private neighbor of $v$.
    We establish that reconstruction from digitally convex sets is equivalent to reconstruction from the set of closed neighborhoods.
\end{abstract}

\maketitle

\section{Introduction}

Throughout this paper, we use the terminology in \cite{Bondy-Murty}, and all graphs are simple and finite. Let $G = (V,E)$ be a graph and $S$ be a vertex subset of $G$.
The \textit{(open) neighborhood} $N_G(v)$ of a vertex $v$ in $G$ is the set of vertices that are adjacent to $v$. We drop the subscript when the host graph is clear from the context. The \textit{closed neighborhood} $N_G[v]$ of $v$ in $G$ is the set $\{v\}\cup N_G(v)$, and the \textit{closed neighborhood} $N_G[S]$ of a set of vertices $S$ in $G$ is $N_G[S] = \cup_{v \in S} N_G[v]$. Given a graph $G$, the {\it open neighborhoods of $G$}, denoted by $\mathcal{N}(G)$, is the multiset $\{N_G(v) : v \in V(G)\}$. Analogously, the {\it closed neighborhoods of $G$}, denoted by $\mathcal{N}[G]$, is the multiset $\{N_G[v] : v \in V(G)\}$.
For a multiset $M$, we define the {\it support of $M$}, denoted by $\supp(M)$, to be the set of elements with positive multiplicity in $M$. In particular, $\text{supp}(\mathcal{N}[G]) = \{M \subseteq V(G) : M = N[v] \text{ for some } v \in V(G)\}$.

Reconstructing a graph from its partial information dates back to the 1940s. A graph $G$ is {\it reconstructible from a parameter $p(G)$} if, for every graph $H$, the equality $p(G)=p(H)$ implies that $G\cong H$. We say a labeled graph $G$ is {\it strongly reconstructible from a parameter $p(G)$} if, for every labeled graph $H$ on the same vertex set as $G$, the equality $p(G)=p(H)$ implies that $G=H$.
Let $G$ be a graph with vertex set $V$ and edge set $E$.
For each vertex $v\in V$, the induced subgraph $G-v$ is called a {\it vertex-deleted subgraph of $G$}. The {\it deck} of a graph $G$ is the multiset $D(G)$ of isomorphism classes of all vertex-deleted subgraphs of $G$. The famous reconstruction conjecture by Kelly \cite{Kelly} and Ulam \cite{Ulam} states that every graph $G$ on at least three vertices is reconstructible from its deck $D(G)$. Harary \cite{Harary} suggested a stronger conjecture saying that every graph $G$ on at least four vertices is reconstructible from the support of its deck $\supp(D(G))$. Both conjectures remain open; see \cite{Bondy,HararySurvey,vonRimscha} for more information.

Beyond these two famous reconstruction conjectures, there has been extensive research on reconstructing graphs from other graph parameters. In particular, when $G$ is a labeled graph, the question extends to whether the original labeled graph can be uniquely determined from a parameter $p(G)$; that is, whether $G$ is strongly reconstructible from $p(G)$.
Some studies focus on determining when a graph $G$ is reconstructible from its open neighborhoods $\mathcal{N}(G)$. First, we note that not all graphs are reconstructible from their open neighborhoods. Consider a 6-vertex cycle $C_6= v_1v_2v_3v_4v_5v_6v_1$ and a disjoint union of two triangles $2K_3= v_1v_3v_5+v_2v_4v_6$ (see Figure~\ref{fig_C_6_and_2K_3}). It is straightforward to check that they are non-isomorphic but have the same open neighborhoods. Aigner and Triesch~\cite{AignerTriesch1993} showed that the decision problem of whether a given multiset is the open neighborhoods of a graph is NP-complete, and they proved that each forest $F$ is reconstructible from the open neighborhoods $\mathcal{N}(F)$. Hammack and Mullican~\cite{Hammack-Mullican} proved that a graph $G$ is reconstructible from its open neighborhoods if and only if $G$ is a cancellation graph.
Moreover, they provided some sufficient conditions for a graph to be reconstructible from its open neighborhoods.

 \begin{figure}[htb]
\hbox to \hsize{
\hfil
\resizebox{9cm}{!}{\tikzset{every picture/.style={line width=0.75pt}} 

\begin{tikzpicture}[x=0.75pt,y=0.75pt,yscale=-1,xscale=1]

\draw  [fill={rgb, 255:red, 0; green, 0; blue, 0 }  ,fill opacity=1 ] (114.12,95.65) .. controls (114.12,93.06) and (116.22,90.96) .. (118.8,90.96) .. controls (121.38,90.96) and (123.48,93.06) .. (123.48,95.65) .. controls (123.48,98.23) and (121.38,100.33) .. (118.8,100.33) .. controls (116.22,100.33) and (114.12,98.23) .. (114.12,95.65) -- cycle ;
\draw  [fill={rgb, 255:red, 0; green, 0; blue, 0 }  ,fill opacity=1 ] (114.12,181.55) .. controls (114.12,178.97) and (116.22,176.87) .. (118.8,176.87) .. controls (121.38,176.87) and (123.48,178.97) .. (123.48,181.55) .. controls (123.48,184.14) and (121.38,186.24) .. (118.8,186.24) .. controls (116.22,186.24) and (114.12,184.14) .. (114.12,181.55) -- cycle ;
\draw  [fill={rgb, 255:red, 0; green, 0; blue, 0 }  ,fill opacity=1 ] (89.32,138.6) .. controls (89.32,136.02) and (91.42,133.92) .. (94,133.92) .. controls (96.58,133.92) and (98.68,136.02) .. (98.68,138.6) .. controls (98.68,141.18) and (96.58,143.28) .. (94,143.28) .. controls (91.42,143.28) and (89.32,141.18) .. (89.32,138.6) -- cycle ;
\draw  [fill={rgb, 255:red, 0; green, 0; blue, 0 }  ,fill opacity=1 ] (163.72,95.65) .. controls (163.72,93.06) and (165.82,90.96) .. (168.4,90.96) .. controls (170.98,90.96) and (173.08,93.06) .. (173.08,95.65) .. controls (173.08,98.23) and (170.98,100.33) .. (168.4,100.33) .. controls (165.82,100.33) and (163.72,98.23) .. (163.72,95.65) -- cycle ;
\draw  [fill={rgb, 255:red, 0; green, 0; blue, 0 }  ,fill opacity=1 ] (163.72,181.55) .. controls (163.72,178.97) and (165.82,176.87) .. (168.4,176.87) .. controls (170.98,176.87) and (173.08,178.97) .. (173.08,181.55) .. controls (173.08,184.14) and (170.98,186.24) .. (168.4,186.24) .. controls (165.82,186.24) and (163.72,184.14) .. (163.72,181.55) -- cycle ;
\draw  [fill={rgb, 255:red, 0; green, 0; blue, 0 }  ,fill opacity=1 ] (188.52,138.6) .. controls (188.52,136.02) and (190.62,133.92) .. (193.2,133.92) .. controls (195.78,133.92) and (197.88,136.02) .. (197.88,138.6) .. controls (197.88,141.18) and (195.78,143.28) .. (193.2,143.28) .. controls (190.62,143.28) and (188.52,141.18) .. (188.52,138.6) -- cycle ;
\draw   (193.2,138.6) -- (168.4,181.55) -- (118.8,181.55) -- (94,138.6) -- (118.8,95.65) -- (168.4,95.65) -- cycle ;

\draw  [fill={rgb, 255:red, 0; green, 0; blue, 0 }  ,fill opacity=1 ] (398.54,69.6) .. controls (398.54,67.02) and (400.64,64.92) .. (403.23,64.92) .. controls (405.81,64.92) and (407.91,67.02) .. (407.91,69.6) .. controls (407.91,72.18) and (405.81,74.28) .. (403.23,74.28) .. controls (400.64,74.28) and (398.54,72.18) .. (398.54,69.6) -- cycle ;
\draw  [fill={rgb, 255:red, 0; green, 0; blue, 0 }  ,fill opacity=1 ] (445.27,123) .. controls (445.27,120.42) and (447.37,118.32) .. (449.95,118.32) .. controls (452.53,118.32) and (454.63,120.42) .. (454.63,123) .. controls (454.63,125.58) and (452.53,127.68) .. (449.95,127.68) .. controls (447.37,127.68) and (445.27,125.58) .. (445.27,123) -- cycle ;
\draw  [fill={rgb, 255:red, 0; green, 0; blue, 0 }  ,fill opacity=1 ] (351.82,123) .. controls (351.82,120.42) and (353.92,118.32) .. (356.5,118.32) .. controls (359.08,118.32) and (361.18,120.42) .. (361.18,123) .. controls (361.18,125.58) and (359.08,127.68) .. (356.5,127.68) .. controls (353.92,127.68) and (351.82,125.58) .. (351.82,123) -- cycle ;
\draw   (403.23,69.6) -- (449.95,123) -- (356.5,123) -- cycle ;

\draw  [fill={rgb, 255:red, 0; green, 0; blue, 0 }  ,fill opacity=1 ] (398.54,152.6) .. controls (398.54,150.02) and (400.64,147.92) .. (403.23,147.92) .. controls (405.81,147.92) and (407.91,150.02) .. (407.91,152.6) .. controls (407.91,155.18) and (405.81,157.28) .. (403.23,157.28) .. controls (400.64,157.28) and (398.54,155.18) .. (398.54,152.6) -- cycle ;
\draw  [fill={rgb, 255:red, 0; green, 0; blue, 0 }  ,fill opacity=1 ] (445.27,206) .. controls (445.27,203.42) and (447.37,201.32) .. (449.95,201.32) .. controls (452.53,201.32) and (454.63,203.42) .. (454.63,206) .. controls (454.63,208.58) and (452.53,210.68) .. (449.95,210.68) .. controls (447.37,210.68) and (445.27,208.58) .. (445.27,206) -- cycle ;
\draw  [fill={rgb, 255:red, 0; green, 0; blue, 0 }  ,fill opacity=1 ] (351.82,206) .. controls (351.82,203.42) and (353.92,201.32) .. (356.5,201.32) .. controls (359.08,201.32) and (361.18,203.42) .. (361.18,206) .. controls (361.18,208.58) and (359.08,210.68) .. (356.5,210.68) .. controls (353.92,210.68) and (351.82,208.58) .. (351.82,206) -- cycle ;
\draw   (403.23,152.6) -- (449.95,206) -- (356.5,206) -- cycle ;

\draw (102,82.4) node [anchor=north west][inner sep=0.75pt]    {$1$};
\draw (74,134.9) node [anchor=north west][inner sep=0.75pt]    {$2$};
\draw (102,187.4) node [anchor=north west][inner sep=0.75pt]    {$3$};
\draw (175.08,187.4) node [anchor=north west][inner sep=0.75pt]    {$4$};
\draw (203,134.9) node [anchor=north west][inner sep=0.75pt]    {$5$};
\draw (175.08,82.4) node [anchor=north west][inner sep=0.75pt]    {$6$};
\draw (398.5,51.4) node [anchor=north west][inner sep=0.75pt]    {$1$};
\draw (336.5,125.4) node [anchor=north west][inner sep=0.75pt]    {$3$};
\draw (458.5,125.4) node [anchor=north west][inner sep=0.75pt]    {$5$};
\draw (458.5,208.4) node [anchor=north west][inner sep=0.75pt]    {$6$};
\draw (336.5,208.4) node [anchor=north west][inner sep=0.75pt]    {$4$};
\draw (398.5,134.4) node [anchor=north west][inner sep=0.75pt]    {$2$};

\end{tikzpicture}}%
\hfil
}
\caption{Two non-isomorphic graphs that have the same set of open neighborhoods.}
\label{fig_C_6_and_2K_3}
\end{figure}
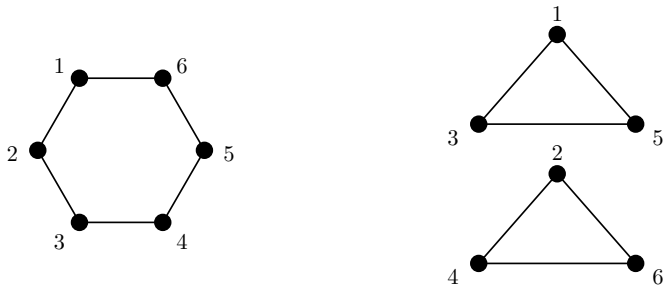

Analogously, it is natural to ask when a graph $G$ is reconstructible from its closed neighborhoods $\mathcal{N}[G]$. We also note that there are infinitely many pairs of non-isomorphic graphs that share the same multiset of closed neighborhoods. For instance, the graphs $K_{3,3}$ and the triangular prism $C_3 \square K_2$, with the vertex labelings shown in Figure \ref{fig_K3,3_and_prism}, have the same multiset of closed neighborhoods.

 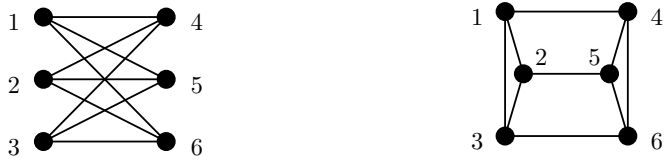
\begin{figure}[htb]
\hbox to \hsize{
\hfil
\resizebox{9cm}{!}{\tikzset{every picture/.style={line width=0.75pt}} 

\begin{tikzpicture}[x=0.75pt,y=0.75pt,yscale=-1,xscale=1]

\draw  [fill={rgb, 255:red, 0; green, 0; blue, 0 }  ,fill opacity=1 ] (98.32,114.2) .. controls (98.32,111.62) and (100.42,109.52) .. (103,109.52) .. controls (105.58,109.52) and (107.68,111.62) .. (107.68,114.2) .. controls (107.68,116.78) and (105.58,118.88) .. (103,118.88) .. controls (100.42,118.88) and (98.32,116.78) .. (98.32,114.2) -- cycle ;
\draw  [fill={rgb, 255:red, 0; green, 0; blue, 0 }  ,fill opacity=1 ] (98.32,182.2) .. controls (98.32,179.62) and (100.42,177.52) .. (103,177.52) .. controls (105.58,177.52) and (107.68,179.62) .. (107.68,182.2) .. controls (107.68,184.78) and (105.58,186.88) .. (103,186.88) .. controls (100.42,186.88) and (98.32,184.78) .. (98.32,182.2) -- cycle ;
\draw  [fill={rgb, 255:red, 0; green, 0; blue, 0 }  ,fill opacity=1 ] (98.32,148.2) .. controls (98.32,145.62) and (100.42,143.52) .. (103,143.52) .. controls (105.58,143.52) and (107.68,145.62) .. (107.68,148.2) .. controls (107.68,150.79) and (105.58,152.88) .. (103,152.88) .. controls (100.42,152.88) and (98.32,150.79) .. (98.32,148.2) -- cycle ;

\draw  [fill={rgb, 255:red, 0; green, 0; blue, 0 }  ,fill opacity=1 ] (165.32,114.2) .. controls (165.32,111.62) and (167.42,109.52) .. (170,109.52) .. controls (172.58,109.52) and (174.68,111.62) .. (174.68,114.2) .. controls (174.68,116.78) and (172.58,118.88) .. (170,118.88) .. controls (167.42,118.88) and (165.32,116.78) .. (165.32,114.2) -- cycle ;
\draw  [fill={rgb, 255:red, 0; green, 0; blue, 0 }  ,fill opacity=1 ] (165.32,182.2) .. controls (165.32,179.62) and (167.42,177.52) .. (170,177.52) .. controls (172.58,177.52) and (174.68,179.62) .. (174.68,182.2) .. controls (174.68,184.78) and (172.58,186.88) .. (170,186.88) .. controls (167.42,186.88) and (165.32,184.78) .. (165.32,182.2) -- cycle ;
\draw  [fill={rgb, 255:red, 0; green, 0; blue, 0 }  ,fill opacity=1 ] (165.32,148.2) .. controls (165.32,145.62) and (167.42,143.52) .. (170,143.52) .. controls (172.58,143.52) and (174.68,145.62) .. (174.68,148.2) .. controls (174.68,150.79) and (172.58,152.88) .. (170,152.88) .. controls (167.42,152.88) and (165.32,150.79) .. (165.32,148.2) -- cycle ;

\draw  [fill={rgb, 255:red, 0; green, 0; blue, 0 }  ,fill opacity=1 ] (350.32,111.2) .. controls (350.32,108.62) and (352.42,106.52) .. (355,106.52) .. controls (357.58,106.52) and (359.68,108.62) .. (359.68,111.2) .. controls (359.68,113.78) and (357.58,115.88) .. (355,115.88) .. controls (352.42,115.88) and (350.32,113.78) .. (350.32,111.2) -- cycle ;
\draw  [fill={rgb, 255:red, 0; green, 0; blue, 0 }  ,fill opacity=1 ] (350.32,179.2) .. controls (350.32,176.62) and (352.42,174.52) .. (355,174.52) .. controls (357.58,174.52) and (359.68,176.62) .. (359.68,179.2) .. controls (359.68,181.78) and (357.58,183.88) .. (355,183.88) .. controls (352.42,183.88) and (350.32,181.78) .. (350.32,179.2) -- cycle ;
\draw  [fill={rgb, 255:red, 0; green, 0; blue, 0 }  ,fill opacity=1 ] (360.32,145.2) .. controls (360.32,142.62) and (362.42,140.52) .. (365,140.52) .. controls (367.58,140.52) and (369.68,142.62) .. (369.68,145.2) .. controls (369.68,147.79) and (367.58,149.88) .. (365,149.88) .. controls (362.42,149.88) and (360.32,147.79) .. (360.32,145.2) -- cycle ;
\draw  [fill={rgb, 255:red, 0; green, 0; blue, 0 }  ,fill opacity=1 ] (417.32,111.2) .. controls (417.32,108.62) and (419.42,106.52) .. (422,106.52) .. controls (424.58,106.52) and (426.68,108.62) .. (426.68,111.2) .. controls (426.68,113.78) and (424.58,115.88) .. (422,115.88) .. controls (419.42,115.88) and (417.32,113.78) .. (417.32,111.2) -- cycle ;
\draw  [fill={rgb, 255:red, 0; green, 0; blue, 0 }  ,fill opacity=1 ] (417.32,179.2) .. controls (417.32,176.62) and (419.42,174.52) .. (422,174.52) .. controls (424.58,174.52) and (426.68,176.62) .. (426.68,179.2) .. controls (426.68,181.78) and (424.58,183.88) .. (422,183.88) .. controls (419.42,183.88) and (417.32,181.78) .. (417.32,179.2) -- cycle ;
\draw  [fill={rgb, 255:red, 0; green, 0; blue, 0 }  ,fill opacity=1 ] (407.32,145.2) .. controls (407.32,142.62) and (409.42,140.52) .. (412,140.52) .. controls (414.58,140.52) and (416.68,142.62) .. (416.68,145.2) .. controls (416.68,147.79) and (414.58,149.88) .. (412,149.88) .. controls (409.42,149.88) and (407.32,147.79) .. (407.32,145.2) -- cycle ;
\draw    (103,114.2) -- (170,114.2) ;
\draw    (103,148.2) -- (170,148.2) ;
\draw    (103,182.2) -- (170,182.2) ;
\draw    (103,114.2) -- (170,148.2) ;
\draw    (103,114.2) -- (170,182.2) ;
\draw    (103,148.2) -- (170,114.2) ;
\draw    (103,182.2) -- (170,114.2) ;
\draw    (103,148.2) -- (170,182.2) ;
\draw    (103,182.2) -- (170,148.2) ;
\draw    (355,111.2) -- (422,111.2) ;
\draw    (355,179.2) -- (422,179.2) ;
\draw    (365,145.2) -- (412,145.2) ;
\draw    (355,111.2) -- (355,179.2) ;
\draw    (422,111.2) -- (422,179.2) ;
\draw    (355,111.2) -- (365,145.2) ;
\draw    (355,179.2) -- (365,145.2) ;
\draw    (412,145.2) -- (422,111.2) ;
\draw    (412,145.2) -- (422,179.2) ;

\draw (82,111.4) node [anchor=north west][inner sep=0.75pt]    {$1$};
\draw (82,145.4) node [anchor=north west][inner sep=0.75pt]    {$2$};
\draw (82,179.4) node [anchor=north west][inner sep=0.75pt]    {$3$};
\draw (182,111.4) node [anchor=north west][inner sep=0.75pt]    {$4$};
\draw (182,145.4) node [anchor=north west][inner sep=0.75pt]    {$5$};
\draw (182,179.4) node [anchor=north west][inner sep=0.75pt]    {$6$};
\draw (335,108.4) node [anchor=north west][inner sep=0.75pt]    {$1$};
\draw (370,130.4) node [anchor=north west][inner sep=0.75pt]    {$2$};
\draw (335,176.4) node [anchor=north west][inner sep=0.75pt]    {$3$};
\draw (433,108.4) node [anchor=north west][inner sep=0.75pt]    {$4$};
\draw (399,130.4) node [anchor=north west][inner sep=0.75pt]    {$5$};
\draw (433,176.4) node [anchor=north west][inner sep=0.75pt]    {$6$};

\end{tikzpicture}}%
\hfil
}
\caption{Two non-isomorphic graphs that have the same set of closed neighborhoods.}
\label{fig_K3,3_and_prism}
\end{figure}

Let $H$ be a graph. A graph is {\it $H$-free} if it does not contain $H$ as an induced subgraph. 
Fomin et al.~\cite{Fomin} studied the complexity of the realizability problem on the multiset of closed neighborhoods.
They showed that for a graph $H \in \{P_k, C_k\}$, where $P_k,C_k$ are paths and cycles on $k$ vertices respectively, given $n$ subsets $\mathcal{S} = \{S_1,S_2,\dots,S_n\}$ of an $n$-element set $V$, it is polynomial-time solvable to determine whether there is an $H$-free graph $G$ with $\mathcal{N}[G] = \mathcal{S}$ for $k \le 4$, and NP-complete for $k > 4$. Their proof for $C_4$-free graphs implies that every $C_4$-free graph is strongly reconstructible from its closed neighborhoods.  For the sake of completeness, we provide an alternative proof of this result in Section~\ref{sec:alter proof}.

\begin{theorem}[Fomin, Kratochv\'il, Lokshtanov, Mancini, and Telle~\cite{Fomin}]
    \label{thm: C4 multiset}
    If $G$ is a $C_4$-free graph, then $G$ is strongly reconstructible from its closed neighborhoods $\mathcal{N}[G]$.
\end{theorem}

We prove the following stronger result.

\begin{theorem}
    \label{thm: C4 set}
    If $G$ is a $C_4$-free graph, then $G$ is strongly reconstructible from $\text{supp}(\mathcal{N}[G])$.
\end{theorem}

This reconstruction problem was initially motivated by a problem arising in digital convexity. A \textit{convexity} $\mathscr{C}$ on a finite set $V$ is a collection of subsets of $V$ such that $\{\emptyset, V\} \subseteq \mathscr{C}$ and $\mathscr{C}$ is closed under intersection. For a graph $G$, a set $S$ of vertices of $G$ is {\it digitally convex} if, for each $v\notin S$, there is a vertex $x$ in $N[v]\setminus N[S]$. We call such a vertex $x$ a {\it private neighbor of $v$ with respect to $S$}. The \textit{digital convexity} $\mathscr{D}(G)$ of a graph $G$ is the collection of all digitally convex sets of $G$. One may easily check that the set of digitally convex sets of $G$ is a convexity of $V(G)$. Lafrance, Oellermann, and Pressey \cite{LAFRANCE2017254} studied when a graph $G$ is strongly reconstructible from its digital convexity $\mathscr{D}(G)$. They showed that every tree $T$ is strongly reconstructible from $\mathscr{D}(T)$, and they also pointed out that their proof methods could be generalized to graphs with girth at least seven.

\begin{theorem}[Lafrance, Oellermann, and Pressey~\cite{LAFRANCE2017254}]
\label{thm:DC&tree}%
Every tree $T$ is strongly reconstructible from its digital convexity $\mathscr{D}(T)$.
\end{theorem}

The following equivalence demonstrates the correspondence between digitally convex sets and closed neighborhoods of vertex subsets.

\begin{proposition}
\label{prop:DC&nbhd}
    Given a graph $G$, a vertex subset $S$ is digitally convex if and only if there is a vertex subset $A$ such that $N[A] = V(G) \setminus S$.
\end{proposition}

\begin{proof}
    If $S$ is digitally convex, then let $A=V(G)\setminus N[S]$. It is straightforward to check that $N[A] = V(G) \setminus S$.

    Conversely, for each vertex $v$ not in $S$, there is a vertex $a\in A$ such that $v\in N[a]$. Because $N[a]\cap S=\emptyset$, we know $a$ is a private neighbor of $v$ with respect to $S$. Hence $S$ is digitally convex. 
\end{proof}

Note that for a graph $G$, the set $\supp(\{N[A] : A \subseteq V(G)\})$ consists of all possible unions of the sets in $\supp(\cN[G])$. Thus, from $\supp(\cN[G])$ one can easily recover $\supp(\{N[A] : A \subseteq V(G)\})$, and hence recover $\mathscr{D}(G)$. However, the converse is not straightforward. In this paper, we establish the following result.

\begin{theorem}
\label{thm:DC=CN-recon}
A graph $G$ is strongly reconstructible from its digital convexity $\mathscr{D}(G)$ if and only if $G$ is strongly reconstructible from $\text{supp}(\mathcal{N}[G])$.
\end{theorem}

Therefore, by Theorem~\ref{thm: C4 set} and \ref{thm:DC=CN-recon}, we improve Lafrance, Oellermann, and Pressey's results.

\begin{theorem}\label{thm:DCS_C4-free}
    If $G$ is a $C_4$-free graph, then $G$ is strongly reconstructible from its digital convexity $\mathscr{D}(G)$.
\end{theorem}

In~\cite{LAFRANCE2017254}, Lafrance et al.\ asked whether the girth of a graph $G$ can be determined from its digital convexity $\mathscr{D}(G)$. 
We resolve this question affirmatively for graphs with girth at least five. 

\begin{corollary}\label{cor:girth_five}
    If $G$ is a graph with girth at least five, then $G$ is strongly reconstructible from its digital convexity $\mathscr{D}(G)$.
\end{corollary}

However, if $G$ contains an induced $C_4$, the digital convexity $\mathscr{D}(G)$ may not uniquely determine the girth of $G$. For instance, as shown in Figure~\ref{fig_K3,3_and_prism}, the graph $K_{3,3}$, which has girth four, and the triangular prism, which has girth three, have the same digital convexity.

The remainder of this paper is organized as follows. In Section~2, we provide an alternative proof of Theorem~\ref{thm: C4 multiset}, along with observations on graphs with the same multiset of closed neighborhoods. Section~3 introduces the notion of union basis, which we apply in Section~4 to prove Theorems~\ref{thm: C4 set} and~\ref{thm:DC=CN-recon}.

\section{An alternative proof of Theorem~\ref{thm: C4 multiset}}\label{sec:alter proof}

Before presenting the proof, we note that Theorem~\ref{thm: C4 multiset} is equivalent to the following.

\begin{theorem}\label{thm: C4 nonconstructable}
    Suppose $G$ and $H$ are two distinct graphs such that $\mathcal{N}[G] = \mathcal{N}[H]$. Then both $G$ and $H$ contain an induced $C_4$.
\end{theorem}

If $\mathcal{N}[G] = \mathcal{N}[H]$ for two graphs $G$ and $H$, then $G$ and $H$ have the same vertex set $V$ and there is a permutation $\sigma$ of $V$ such that $\sigma(v) = u$ implies $N_G[v] = N_H[u]$. 
Define the \textit{orbit of $v$ with respect to $\sigma$}, denoted by $\mathcal{O}_{\sigma}(v)$ (or simply $\mathcal{O}(v)$ when $\sigma$ is understood from context), as the set $\{v, \sigma(v), \sigma^2(v), \dots\}$. 
Note that the orbits partition the vertex set $V$. Because $V$ is finite, we have $|\mathcal{O}_\sigma(v)|<\infty$ and $\mathcal{O}_\sigma(v) = \mathcal{O}_{\sigma^{-1}}(v)$ for each $v\in V$. First we make the following elementary observation.

\begin{observation}\label{obs: edge transit}
    Let $G$ and $H$ be two graphs on the same vertex set $V$ such that $\mathcal{N}[G] = \mathcal{N}[H]$. Suppose that $\sigma$ is a permutation of $V$ such that $\sigma(v) = u$ implies $N_G[v] = N_H[u]$. Then the followings hold.
    \begin{enumerate}
        \item[(\romannum{1})] If $a\in N_G[b]$, then $a\in N_H[\sigma(b)]$, and
        \item[(\romannum{2})] if $c\in N_H[d]$, then $c\in N_G[\sigma^{-1}(d)]$.
    \end{enumerate}
\end{observation}

\begin{lemma}\label{lem:same number of edges}
    Suppose $G$ and $H$ are two graphs such that $\mathcal{N}[G] = \mathcal{N}[H]$. Then $|E(G)|=|E(H)|$.
\end{lemma}

\begin{proof}
    It is straightforward to verify that the number of edges in both graphs is
    \begin{align*}
        |E(G)|
        =\frac{1}{2}\sum_{v\in V(G)} \deg_G(v)
        &=\frac{1}{2}\sum_{M\in\mathcal{N}[G]}(|M|-1) \\
        &=\frac{1}{2}\sum_{M\in\mathcal{N}[H]}(|M|-1)
        =\frac{1}{2}\sum_{u\in V(H)} \deg_H(u)
        =|E(H)|.\qedhere
    \end{align*}
\end{proof}

The next lemma characterizes the structure inside each orbit.

\begin{lemma}\label{lem: induced clique}
    Let $G$ and $H$ be two graphs on the same vertex set $V$ such that $\mathcal{N}[G] = \mathcal{N}[H]$. 
    Suppose that $\sigma$ is a permutation of $V$ such that $\sigma(v) = u$ implies $N_G[v] = N_H[u]$, and $\mathcal{O}(v)$ is the orbit of a vertex $v$ with respect to $\sigma$. Then $\mathcal{O}(v)$ induces a clique in both $G$ and $H$.
\end{lemma}

\begin{proof}
We show that, for an arbitrary integer $i$ and a positive integer $k$, the vertex $\sigma^i(v)$ is in the closed neighborhood of $\sigma^{i+k}(v)$ in both $G$ and $H$. 
    We apply induction on $k$ for each $i$. For the base case $k = 1$, since $N_G[\sigma^i(v)]=N_H[\sigma^{i+1}(v)]$ contains both $\sigma^i(v)$ and $\sigma^{i+1}(v)$, it is clear that $\sigma^i(v)\in N_G[\sigma^{i+1}(v)]$ and $\sigma^i(v)\in N_H[\sigma^{i+1}(v)]$. 
    
    Now suppose that, for each integer $i$, the vertex $\sigma^i(v)$ is in the closed neighborhood of $\sigma^{i+j}(v)$ in both $G$ and $H$ for all $j<k$. Since $\sigma^i(v)\in N_G[\sigma^{i+k-1}(v)]= N_H[\sigma^{i+k}(v)]$, we know $\sigma^i(v)$ is in the closed neighborhood of $\sigma^{i+k}(v)$ in $H$. 
    Moreover, since $\sigma^{i+1}(v) \in N_H[\sigma^{i+k}(v)]$ by the inductive hypothesis, we know $\sigma^{i+k}(v)\in N_H[\sigma^{i+1}(v)]$. By Observation~\ref{obs: edge transit}, $\sigma^{i+k}(v)\in N_G[\sigma^{i}(v)]$, and equivalently, $\sigma^{i}(v) \in N_G[\sigma^{i+k}(v)]$. Therefore, since $\mathcal{O}(v)$ is finite, every two distinct vertices in $\mathcal{O}(v)$ are adjacent. Thus, $\mathcal{O}(v)$ induces a clique in both $G$ and $H$.
\end{proof}

\noindent{\it Proof of Theorem~\ref{thm: C4 nonconstructable}.}
    Suppose, on the contrary, that at least one of $G$ and $H$ does not contain an induced $C_4$. Without loss of generality, we may assume that $H$ is $C_4$-free. Since $G$ and $H$ are distinct graphs on the same set of vertices and, by Lemma~\ref{lem:same number of edges}, have the same number of edges, $G$ is not a proper subgraph of $H$.
    It follows that $G$ has an edge $ab$ that is not in $H$. Let $\sigma$ be a permutation of $V=V(G)=V(H)$ such that $\sigma(v)=u$ implies that $N_G[v]=N_H[u]$. 
    By Lemma~\ref{lem: induced clique}, it is clear that $\mathcal{O}(a)$ and $\mathcal{O}(b)$ are disjoint.
    We show by induction that 
    \begin{sublemma}\label{sublem: two orbits are joined}
        in the graph $H$, for every pair of positive integers $i$ and $j$, 
        \begin{enumerate}
            \item[(\romannum{1})] $\sigma^{i}(a)$ is adjacent to $b$,
            \item[(\romannum{2})] $\sigma^j(b)$ is adjacent to $a$, and
            \item[(\romannum{3})] $\sigma^{i}(a)$ is adjacent to $\sigma^j(b)$.
        \end{enumerate}
    \end{sublemma}
     
    For the base case $i=1$, since $a$ is adjacent to $b$ in $G$, by Observation~\ref{obs: edge transit}, $\sigma(a)$ is adjacent to $b$ and $\sigma(b)$ is adjacent to $a$ in $H$.
    Moreover, since $a$ is not adjacent to $b$ in $H$ and, by Lemma~\ref{lem: induced clique}, $\mathcal{O}(a)$, $\mathcal{O}(b)$ form cliques, we know $\sigma(a)$ is adjacent to $\sigma(b)$ in $H$, otherwise there is an induced $C_4$ in $H$ on the vertices $a, \sigma(a), b, \sigma(b)$.
    
    Now suppose that the claim holds for all $i<s$ and $j<t$. 
    Since $\sigma^{s-1}(a)$ is adjacent to $\sigma(b)$ in $H$, by Observation~\ref{obs: edge transit}, $\sigma^{s-1}(a)$ is adjacent to $b$ in $G$. 
    Similarly, $\sigma^{t-1}(b)$ is adjacent to $a$ in $G$. 
    Therefore, by Observation~\ref{obs: edge transit}, $\sigma^{s}(a)$ is adjacent to $b$ and $\sigma^{t}(b)$ is adjacent to $a$ in $H$. Since $a$ is not adjacent to $b$ in $H$ and $H$ does not contain an induced $C_4$, we know $\sigma^{s}(a)$ is adjacent to $\sigma^{t}(b)$ in $H$. Thus,~\ref{sublem: two orbits are joined} holds.\\
    
    However, since both orbits are finite, there are two positive integers $x,y$ such that $\sigma^x(a)=a$ and $\sigma^y(b)=b$. By~\ref{sublem: two orbits are joined}, $\sigma^x(a)$ is adjacent to $\sigma^y(b)$ in $H$, contradicting the fact that $a$ and $b$ are not adjacent in $H$.
    \qed

\section{Union of closed neighborhoods}

Recall that, given a graph $G$, the set $\supp(\{N[A] : A \subseteq V(G)\})$ consists of all possible unions of the sets in $\supp(\mathcal{N}[G])$. We denote the set $\supp(\{N[A] : A \subseteq V(G)\})$ by $U(\mathcal{N}[G])$.

\begin{observation}\label{obs:supp to U} 
    Suppose $G$ is a graph. Then $U(\mathcal{N}[G])$ can be determined from $\supp(\mathcal{N}[G])$.
\end{observation}

From $U(\mathcal{N}[G])$, we can determine the inclusion relation between $N[A]$ and $N[B]$ for each pair $(A,B)$ of vertex subsets.

\begin{lemma}
    \label{lem: inclusion CN}
    Let $G$ be a graph and let $A$ and $B$ be sets of vertices of $G$.
    Then $N[A] \subseteq N[B]$ if and only if $A \cap M \ne \emptyset$ implies $B \cap M \ne \emptyset$ for all $M \in U(\mathcal{N}[G])$.
\end{lemma}

\begin{proof}
    Suppose that $N[A] \subseteq N[B]$, and  $M$ is a set in $U(\mathcal{N}[G])$ such that $A \cap M \ne \emptyset$.
    Let $S$ be a set of vertices in $G$ such that $M = N[S]$.
    Suppose that $a$ is a vertex in $A \cap M$. Then there is a vertex $x \in S$ such that $a \in A \cap N[x]$.
    Then, $x \in N[a] \subseteq N[A] \subseteq N[B]$. Clearly, there is a vertex $b \in B$ such that $x \in N[b]$ and thus $b \in N[x] \subseteq M$.
    Therefore, $B \cap M \ne \emptyset$.

    Suppose that $N[A] \not\subseteq N[B]$.
    Let $x \in N[A] \setminus N[B]$ and $M = N[x]$.
    Then $A \cap M \ne \emptyset$ and $B \cap M = \emptyset$.
\end{proof}

As a direct consequence, from $U(\mathcal{N}[G])$ we can recognize whether the closed neighborhoods of two vertex subsets are equal or not.

\begin{corollary}
    \label{coro: equal CN}
    Let $G$ be a graph, and let $A$ and $B$ be sets of vertices of $G$.
    Then $N[A] = N[B]$ if and only if 
    $$A \cap M \ne \emptyset \Longleftrightarrow B \cap M \ne \emptyset$$
    for all $M \in U(\mathcal{N}[G])$.
\end{corollary}

Given a collection $\mathcal{X}$ of sets, we denote the set $\{\bigcup\limits_{X \in \mathcal{X}'} X : \mathcal{X}' \subseteq \mathcal{X}\}$ by  $\text{span} (\mathcal{X})$.
Given a finite collection $\mathcal{A}$ of sets, we say that $\mathcal{A}$ is {\it spanned by} another collection $\mathcal{B} $ if $\mathcal{A} \subseteq \text{span} (\mathcal{B})$.
Note that the relation is transitive, which means that if $\mathcal{A}$ is spanned by $\mathcal{B}$ and $\mathcal{B}$ is spanned by $\mathcal{C}$, then $\mathcal{A}$ is spanned by $\mathcal{C}$.
A subcollection $\mathcal{B}$ of $\mathcal{A}$ is called a {\it union basis of $\mathcal{A}$} if $\mathcal{A}$ is spanned by $\mathcal{B}$ and is not spanned by any proper subset of $\mathcal{B}$.

\begin{proposition}
    \label{prop: unique union base}
    There is a unique union basis for every finite collection $\mathcal{A}$.
\end{proposition}

\begin{proof}
    We construct a sequence $\mathcal{A}_0, \mathcal{A}_1\dots, \mathcal{A}_k$ of subcollections of $\mathcal{A}$ as follows.
    \begin{enumerate}
        \item[(\romannum{1})] Let $\mathcal{A}_0 = \mathcal{A}$.
        \item[(\romannum{2})] For all $i\geq 0$, if there is a set $A \in \text{span}(\mathcal{A}_i \setminus \{A\})$, then let $\mathcal{A}_{i+1} = \mathcal{A}_i \setminus \{A\}$, otherwise
        \item[(\romannum{3})] let $k=i$, and terminate the sequence.
    \end{enumerate}
    Since $\mathcal{A}$ is finite, the sequence will terminate in a finite number of steps. It is straightforward to see that $\mathcal{A}_k$ is a union basis of $\mathcal{A}$.

    To prove the uniqueness, suppose that $\mathcal{B}_1$ and $\mathcal{B}_2$ are two distinct union bases of $\mathcal{A}$.
    Without loss of generality, we assume that there is a set $A \in \mathcal{B}_1 \setminus \mathcal{B}_2$.
    Since $\mathcal{B}_2$ is a union basis of $\mathcal{A}$, there is a subcollection $\{B_1,B_2,\dots,B_k\}$ of $\mathcal{B}_2$ such that $A = B_1\cup B_2\cup\dots\cup B_k$.
    Note that $A \notin \mathcal{B}_2$, so $B_i \subsetneq A$ for all $i\in\{1,2,\dots,k\}$.
    Moreover, since $\mathcal{B}_1$ is a union basis of $\mathcal{A}$, and $\{B_1,B_2,\dots,B_k\}\subseteq\mathcal{A}$, we know $\{B_1,B_2,\dots,B_k\}$ is spanned by $\mathcal{B}_1 \setminus \{A\}$. Therefore, $A\in \text{span}(\mathcal{B}_1 \setminus \{A\})$ and $\mathcal{B}_1$ is spanned by $\mathcal{B}_1 \setminus \{A\}$. By transitivity, $\mathcal{A}$ is spanned by $\mathcal{B}_1 \setminus \{A\}$, a contradiction.
    Therefore, the union basis is unique.
\end{proof}

A vertex subset $S$ of a graph $G$ is called {\it a set of base vertices of $G$} if $\{N[v]: v \in S\}$ is the union basis of $\supp(\mathcal{N}[G])$.

\begin{lemma}
    \label{lem: base vtx}
    Given $U(\mathcal{N}[G])$ of a graph $G$, one can find a set of base vertices of $G$.
\end{lemma}

\begin{proof}
    Given $U(\mathcal{N}[G])$, by Corollary~\ref{coro: equal CN}, for each pair of vertex $v$ and vertex subset $A$ of $G$, we can determine whether $N[v] = N[A]$ or not. We construct a sequence $S_0,S_1,\dots,S_k$ as follows.
    \begin{enumerate}
        \item[(\romannum{1})] Let $S_0=V(G)$.
        \item[(\romannum{2})] For all $i\geq 0$, if $S_i$ contains a vertex $v$ such that $N[v] = N[A]$ for some $A \subseteq S_i \setminus \{v\}$, then let $S_{i+1}=S_i\setminus\{v\}$, otherwise
        \item[(\romannum{3})] let $k=i$, and terminate the sequence.
    \end{enumerate}
    It is straightforward to see that $S_k$ is a set of base vertices.
\end{proof}

The next result follows immediately from the previous lemma and Observation~\ref{obs:supp to U}.

\begin{corollary}
    Given $\supp(\mathcal{N}[G])$ of a graph $G$, one can find a set of base vertices of $G$.
\end{corollary}

We remark that a graph can have different sets of base vertices, despite the uniqueness of the union basis,
since there may exist different vertices with the same closed neighborhoods that are in the union basis.
For example, if $G$ is a complete graph $K_n$ for some $n\geq 2$, then every single-vertex set is a set of base vertex, and there are $n$ such sets.

\section{Proof of Theorems~\ref{thm: C4 set} and~\ref{thm:DC=CN-recon}}

In this section, we prove Theorems~\ref{thm: C4 set} and~\ref{thm:DC=CN-recon}.  We first define some concepts and notation that we will use. Let $G$ be a graph and let $v$ be a vertex of $G$. To \textit{blow up $v$ into a clique of size $k$}, we first delete $v$ from $G$, 
and then add a set $\{v_1, v_2, \dots, v_k\}$ of new vertices that are mutually adjacent. 
For each $i \in \{1,2,\dots,k\}$, we make $v_i$ adjacent to every vertex in $N_G(v)$. For vertex subsets $S_1$ and $S_2$ of a graph $G$, the set $\mathit{E_G[S_1,S_2]}$ is the set of edges in $G$ that have one endpoint in $S_1$ and the other endpoint in $S_2$.

We prove the following result.

\begin{theorem}\label{thm:induced-subgraph-closed class}
     Suppose that $\mathcal{G}$ is an induced-subgraph-closed class of graphs, and $G$ is strongly reconstructible from $\cN[G]$ for all $G\in\mathcal{G}$. Then $G$ is strongly reconstructible from $\supp(\cN[G])$ for all $G\in\mathcal{G}$.
\end{theorem}

\begin{proof}
    Given $\supp(\mathcal{N}[G])$ of a graph $G$ in $\mathcal{G}$, by Corollary~\ref{coro: equal CN}, for each pair of vertices $u$ and $v$, we can determine whether $N_G[u] = N_G[v]$ or not. For each pair of vertices $u$ and $v$, we define the equivalence relation $u\sim v$ if $N_G[u] = N_G[v]$. Let $V_1,V_2,\dots,V_m$ be the equivalence classes of $V(G)$ with respect to $\sim$. It follows that

    \begin{sublemma}\label{sublem:supp determines Vi}
        $V_1,V_2,\dots,V_m$ are determined from $\supp(\cN[G])$.
    \end{sublemma}
    
    Clearly, in $G$, every such equivalence class induces a clique. Note that if a set $A$ in $\supp(\mathcal{N}[G])$ intersects $V_i$ for some $i\in\{1,2,\dots,m\}$, then $V_i\subseteq A$. For each $A\in \supp(\mathcal{N}[G])$, we define
    \[\overline{A}=\{V_i:V_i\subseteq A \text{ and } i\in\{1,2,\dots,m\}\},\]
    and let $\mathcal{S}$ be the collection $\{\overline{A}:A\in \supp(\mathcal{N}[G])\}$. Clearly,
    \begin{sublemma}\label{sublem:supp determines S}
        $\mathcal{S}$ is determined from $\supp(\cN[G])$.
    \end{sublemma}
    
    Note that, for each $i\in\{1,2,\dots,m\}$ and each pair of vertices $a$ and $b$ in $V_i$, the closed neighborhood $N_G[a]$ of $a$ equals the closed neighborhood $N_G[b]$ of $b$. Therefore, $E_G[V_i,V_j]$
    is either complete or empty for all $j\in \{1,2,\dots,m\}\setminus\{i\}$. 
    Let $G'$ be the graph on vertex set $\{V_1,V_2,\dots,V_m\}$ where $V_i$ and $V_j$ are adjacent in $G'$ if and only if $E_G[V_i,V_j]$ is complete.

    Next we show that

    \begin{sublemma}\label{sublem:S=N[G']}
        $\mathcal{S}=\cN[G']$.
    \end{sublemma}
    Note that $\mathcal{S}$ equals the set $\{\overline{N_G[v]}: v\in V(G)\}$. Moreover, $\overline{N_G[x]}=\overline{N_G[y]}$ if $x\sim y$. For each $i \in \{1,2,\dots,m\}$, choose a representative vertex $v_i$ from the equivalence class $V_i$. It suffices to show that $\overline{N_G[v_i]}=N_{G'}[V_i]$.

    Suppose $V_j$ is adjacent to $V_i$ in $G'$. Then there is a vertex $x \in V_i$ adjacent to a vertex $y \in V_j$ in $G$. Hence, $E_G[V_i,V_j]$ is complete, $V_j \subseteq N_G[v_i]$, and so $V_j \in \overline{N_G[v_i]}$. Moreover, $V_i\subseteq N_G[v_i]$, and $V_i \in \overline{N_G[v_i]}$. Thus, $N_{G'}[V_i]\subseteq \overline{N_G[v_i]}$.
    
    Conversely, suppose $V_j\in \overline{N_G[v_i]}$. If $j=i$, then $V_j=V_i\in N_{G'}[V_i]$. Otherwise $j\neq i$, and there is a vertex $x$ in $V_j$ that is adjacent to a vertex $y$ in $V_i$. Hence $V_j$ is adjacent to $V_i$ in $G'$, so $V_j\in N_{G'}[V_i]$. Therefore,~\ref{sublem:S=N[G']} holds.\\
    
    Note that $G'$ can be obtained from $G$ by deleting all but one vertex in each equivalence class with respect to~$\sim$. Since $G$ belongs to an induced-subgraph-closed class~$\mathcal{G}$, it follows that $G' \in \mathcal{G}$ as well. Hence $G'$ is reconstructible from~$\mathcal{S}$, which equals~$\mathcal{N}[G']$. To obtain $G$ from $G'$, we blow up each vertex~$v_i$ into a clique of size~$|V_i|$ and label the new vertices by the elements of~$V_i$. By~\ref{sublem:supp determines Vi} and~\ref{sublem:supp determines S}, we conclude that $G$ is reconstructible from $\supp(\cN[G])$.
\end{proof}

Since $C_4$-free graphs form an induced-subgraph-closed class, Theorem~\ref{thm: C4 set} follows immediately from Theorems~\ref{thm: C4 multiset} and~\ref{thm:induced-subgraph-closed class}. The following example illustrates the reconstruction process of the $C_4$-free graph $G$ shown in Figure~\ref{fig_example_1}.

 \begin{figure}[htb]
\hbox to \hsize{
\hfil
\resizebox{4.5cm}{!}{\tikzset{every picture/.style={line width=0.75pt}} 

\begin{tikzpicture}[x=0.75pt,y=0.75pt,yscale=-1,xscale=1]

\draw  [fill={rgb, 255:red, 0; green, 0; blue, 0 }  ,fill opacity=1 ] (482.75,71.6) .. controls (482.75,69.7) and (484.3,68.15) .. (486.2,68.15) .. controls (488.1,68.15) and (489.65,69.7) .. (489.65,71.6) .. controls (489.65,73.5) and (488.1,75.05) .. (486.2,75.05) .. controls (484.3,75.05) and (482.75,73.5) .. (482.75,71.6) -- cycle ;
\draw  [fill={rgb, 255:red, 0; green, 0; blue, 0 }  ,fill opacity=1 ] (401.55,133) .. controls (401.55,131.1) and (403.1,129.55) .. (405,129.55) .. controls (406.9,129.55) and (408.45,131.1) .. (408.45,133) .. controls (408.45,134.9) and (406.9,136.45) .. (405,136.45) .. controls (403.1,136.45) and (401.55,134.9) .. (401.55,133) -- cycle ;
\draw  [fill={rgb, 255:red, 0; green, 0; blue, 0 }  ,fill opacity=1 ] (479.75,98.6) .. controls (479.75,96.7) and (481.3,95.15) .. (483.2,95.15) .. controls (485.1,95.15) and (486.65,96.7) .. (486.65,98.6) .. controls (486.65,100.5) and (485.1,102.05) .. (483.2,102.05) .. controls (481.3,102.05) and (479.75,100.5) .. (479.75,98.6) -- cycle ;
\draw  [fill={rgb, 255:red, 0; green, 0; blue, 0 }  ,fill opacity=1 ] (444.46,186.35) .. controls (444.46,184.45) and (446,182.91) .. (447.9,182.91) .. controls (449.81,182.91) and (451.35,184.45) .. (451.35,186.35) .. controls (451.35,188.26) and (449.81,189.8) .. (447.9,189.8) .. controls (446,189.8) and (444.46,188.26) .. (444.46,186.35) -- cycle ;
\draw    (405,133) -- (447.9,186.35) ;
\draw    (486.2,71.6) -- (405,133) ;
\draw    (447.9,186.35) -- (510.7,154) ;
\draw    (405,133) -- (483.2,98.6) ;
\draw  [fill={rgb, 255:red, 0; green, 0; blue, 0 }  ,fill opacity=1 ] (452.75,216.6) .. controls (452.75,214.7) and (454.3,213.15) .. (456.2,213.15) .. controls (458.1,213.15) and (459.65,214.7) .. (459.65,216.6) .. controls (459.65,218.5) and (458.1,220.05) .. (456.2,220.05) .. controls (454.3,220.05) and (452.75,218.5) .. (452.75,216.6) -- cycle ;
\draw  [fill={rgb, 255:red, 0; green, 0; blue, 0 }  ,fill opacity=1 ] (435.4,204.6) .. controls (435.4,202.7) and (436.95,201.15) .. (438.85,201.15) .. controls (440.75,201.15) and (442.3,202.7) .. (442.3,204.6) .. controls (442.3,206.5) and (440.75,208.05) .. (438.85,208.05) .. controls (436.95,208.05) and (435.4,206.5) .. (435.4,204.6) -- cycle ;
\draw  [fill={rgb, 255:red, 0; green, 0; blue, 0 }  ,fill opacity=1 ] (507.25,154) .. controls (507.25,152.1) and (508.8,150.55) .. (510.7,150.55) .. controls (512.6,150.55) and (514.15,152.1) .. (514.15,154) .. controls (514.15,155.9) and (512.6,157.45) .. (510.7,157.45) .. controls (508.8,157.45) and (507.25,155.9) .. (507.25,154) -- cycle ;
\draw    (405,133) -- (510.7,154) ;
\draw    (483.2,98.6) -- (510.7,154) ;
\draw    (510.7,154) -- (486.2,71.6) ;
\draw    (485.7,72.55) -- (483.2,98.6) ;
\draw    (405,133) -- (438.85,204.6) ;
\draw    (405,133) -- (456.2,216.6) ;
\draw    (510.7,154) -- (456.2,216.6) ;
\draw    (510.7,154) -- (438.85,204.6) ;
\draw    (447.9,186.35) -- (456.2,216.6) ;
\draw    (447.9,186.35) -- (438.85,204.6) ;
\draw    (438.85,204.6) -- (456.2,216.6) ;
\draw  [dash pattern={on 4.5pt off 4.5pt}] (427.05,201.48) .. controls (427.05,187.67) and (438.25,176.48) .. (452.05,176.48) .. controls (465.86,176.48) and (477.05,187.67) .. (477.05,201.48) .. controls (477.05,215.28) and (465.86,226.48) .. (452.05,226.48) .. controls (438.25,226.48) and (427.05,215.28) .. (427.05,201.48) -- cycle ;
\draw  [dash pattern={on 4.5pt off 4.5pt}] (459.45,85.57) .. controls (459.45,71.77) and (470.64,60.57) .. (484.45,60.57) .. controls (498.26,60.57) and (509.45,71.77) .. (509.45,85.57) .. controls (509.45,99.38) and (498.26,110.57) .. (484.45,110.57) .. controls (470.64,110.57) and (459.45,99.38) .. (459.45,85.57) -- cycle ;
\draw  [fill={rgb, 255:red, 0; green, 0; blue, 0 }  ,fill opacity=1 ] (352.75,131.6) .. controls (352.75,129.7) and (354.3,128.15) .. (356.2,128.15) .. controls (358.1,128.15) and (359.65,129.7) .. (359.65,131.6) .. controls (359.65,133.5) and (358.1,135.05) .. (356.2,135.05) .. controls (354.3,135.05) and (352.75,133.5) .. (352.75,131.6) -- cycle ;
\draw    (356.2,131.6) -- (405,133) ;

\draw (393,114.4) node [anchor=north west][inner sep=0.75pt]    {$1$};
\draw (345,112.4) node [anchor=north west][inner sep=0.75pt]    {$5$};
\draw (481,48.4) node [anchor=north west][inner sep=0.75pt]    {$6$};
\draw (516.15,162.4) node [anchor=north west][inner sep=0.75pt]    {$3$};
\draw (461.65,225) node [anchor=north west][inner sep=0.75pt]    {$4$};
\draw (471,115.4) node [anchor=north west][inner sep=0.75pt]    {$2$};
\draw (442,163.4) node [anchor=north west][inner sep=0.75pt]    {$8$};
\draw (412,201.4) node [anchor=north west][inner sep=0.75pt]    {$7$};

\end{tikzpicture}}%
\hfil
}
\caption{A $C_4$-free graph $G$.}
\label{fig_example_1}
\end{figure}
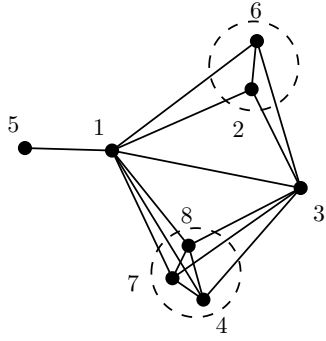

\begin{example} We are given the support $\supp(\mathcal{N}[G])$, 
\[
\big\{\{1,2,3,4,5,6,7,8\},\ \{1,2,3,6\},\ \{1,2,3,4,6,7,8\},\ \{1,3,4,7,8\},\ \{1,5\}\big\},
\]
of the closed neighborhoods of a $C_4$-free graph $G$, and we proceed to reconstruct $G$ as follows.

\begin{enumerate}
    \item[(\romannum{1})] By Corollary~\ref{coro: equal CN}, for every pair of vertices $x$ and $y$ in $G$, we can determine if they have the same set of closed neighborhoods. Let $x\sim y$, if $N_G[x]=N_G[y]$. We partition $V(G)$ into the following equivalence classes,
    \[\{1\},\ \{2,6\},\ \{3\},\ \{4,7,8\},\ \{5\}.\]
    \item[(\romannum{2})] Let $V_1=\{1\}, V_2=\{2,6\}, V_3=\{3\}, V_4=\{4,7,8\},$ and $V_5=\{5\}$. From $\supp(\mathcal{N}[G])$, we define the collection $\mathcal{S}$ as 
    \[\hspace*{1.1cm}\big\{\{V_1,V_2,V_3,V_4,V_5\},\ \{V_1,V_2,V_3\},\ \{V_1,V_2,V_3,V_4\},\ \{V_1,V_3,V_4\},\ \{V_1,V_5\}\big\}.\] We know $\mathcal{S}$ equals $\cN[G']$ for some $C_4$-free graph $G$.
    \item[(\romannum{3})] By Theorem~\ref{thm: C4 multiset}, we can reconstruct $G'$ from $\mathcal{S}$ as shown in Figure~\ref{fig_example_2}.

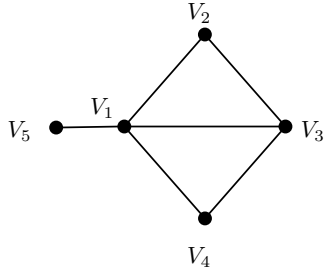
\begin{figure}[htb]
\hbox to \hsize{
\hfil
\resizebox{4.5cm}{!}{\tikzset{every picture/.style={line width=0.75pt}} 

\begin{tikzpicture}[x=0.75pt,y=0.75pt,yscale=-1,xscale=1]

\draw  [fill={rgb, 255:red, 0; green, 0; blue, 0 }  ,fill opacity=1 ] (185.4,70.6) .. controls (185.4,68.7) and (186.95,67.15) .. (188.85,67.15) .. controls (190.75,67.15) and (192.3,68.7) .. (192.3,70.6) .. controls (192.3,72.5) and (190.75,74.05) .. (188.85,74.05) .. controls (186.95,74.05) and (185.4,72.5) .. (185.4,70.6) -- cycle ;
\draw  [fill={rgb, 255:red, 0; green, 0; blue, 0 }  ,fill opacity=1 ] (139.55,123) .. controls (139.55,121.1) and (141.1,119.55) .. (143,119.55) .. controls (144.9,119.55) and (146.45,121.1) .. (146.45,123) .. controls (146.45,124.9) and (144.9,126.45) .. (143,126.45) .. controls (141.1,126.45) and (139.55,124.9) .. (139.55,123) -- cycle ;
\draw  [fill={rgb, 255:red, 0; green, 0; blue, 0 }  ,fill opacity=1 ] (231.25,123) .. controls (231.25,121.1) and (232.8,119.55) .. (234.7,119.55) .. controls (236.6,119.55) and (238.15,121.1) .. (238.15,123) .. controls (238.15,124.9) and (236.6,126.45) .. (234.7,126.45) .. controls (232.8,126.45) and (231.25,124.9) .. (231.25,123) -- cycle ;
\draw  [fill={rgb, 255:red, 0; green, 0; blue, 0 }  ,fill opacity=1 ] (185.46,175.35) .. controls (185.46,173.45) and (187,171.91) .. (188.9,171.91) .. controls (190.81,171.91) and (192.35,173.45) .. (192.35,175.35) .. controls (192.35,177.26) and (190.81,178.8) .. (188.9,178.8) .. controls (187,178.8) and (185.46,177.26) .. (185.46,175.35) -- cycle ;
\draw    (143,123) -- (188.9,175.35) ;
\draw    (188.8,70.65) -- (234.7,123) ;
\draw    (234.7,123) -- (188.9,175.35) ;
\draw    (143,123) -- (188.85,70.6) ;
\draw    (143,123) -- (234.7,123) ;
\draw    (143.05,122.95) -- (104.2,123.6) ;
\draw  [fill={rgb, 255:red, 0; green, 0; blue, 0 }  ,fill opacity=1 ] (100.75,123.6) .. controls (100.75,121.7) and (102.3,120.15) .. (104.2,120.15) .. controls (106.1,120.15) and (107.65,121.7) .. (107.65,123.6) .. controls (107.65,125.5) and (106.1,127.05) .. (104.2,127.05) .. controls (102.3,127.05) and (100.75,125.5) .. (100.75,123.6) -- cycle ;

\draw (122,105.4) node [anchor=north west][inner sep=0.75pt]    {$V_{1}$};
\draw (177.25,51.4) node [anchor=north west][inner sep=0.75pt]    {$V_{2}$};
\draw (242,118.4) node [anchor=north west][inner sep=0.75pt]    {$V_{3}$};
\draw (177.25,190.4) node [anchor=north west][inner sep=0.75pt]    {$V_{4}$};
\draw (75,118.4) node [anchor=north west][inner sep=0.75pt]    {$V_{5}$};

\end{tikzpicture}}%
\hfil
}
\caption{$G'$ can be reconstructed from $\mathcal{S}$, which equals $\cN[G']$.}
\label{fig_example_2}
\end{figure}

    \item[(\romannum{4})] To complete the reconstruction of $G$, for each $i \in \{1,2,3,4,5\}$, we blow up $V_i$ into a clique of size $|V_i|$ and label the vertices by the elements of $V_i$. 
    The resulting graph is $G$. \hfill \qed

\end{enumerate}
    
\end{example}

\noindent It remains to complete the proof of Theorem~\ref{thm:DC=CN-recon}.

\begin{proof}   [Proof of Theorem~\ref{thm:DC=CN-recon}]
   Let $G$ be a graph with vertex set $V$. By Proposition~\ref{prop:DC&nbhd}, $U(\mathcal{N}[G])$ is precisely the set of complements of the sets in $\mathscr{D}(G)$.
    Suppose that $G$ is reconstructible from $\mathscr{D}(G)$. Given $\supp(\mathcal{N}[G])$, by Observation~\ref{obs:supp to U}, we can determine $U(\mathcal{N}[G])$ and hence determine $\mathscr{D}(G)$.
    Therefore, $G$ is reconstructible from $\supp(\mathcal{N}[G])$.
    
    Conversely, suppose that $G$ is reconstructible from $\supp(\mathcal{N}[G])$.
    We argue by induction on $|V(G)|$.
    For $|V(G)| = 1$, the statement is trivially true.
    Suppose that the graph $G$ has at least two vertices and the statement holds for every graph $G'$ of order less than $G$.
    
    Given $\mathscr{D}(G)$,  we can determine $U(\mathcal{N}[G])$. By Lemma~\ref{lem: base vtx}, we can find a set $S$ of base vertices of $G$. Let $G'$ be the subgraph of $G$ induced on $S$.
    If $S = V(G)$, then $G' = G$ and $\supp(\mathcal{N}[G]) = \mathcal{N}[G] = \{N[v] : v \in S\}$.
    Since $G$ is reconstructible from $\supp(\mathcal{N}[G])$ by assumption, then $G$ is reconstructible.
    If $S \subsetneq V(G)$, then $G'$ is reconstructible from $\supp(\mathcal{N}[{G'}]) = \{N[v] : v \in S\}$.
    For each vertex $v \in V(G) \setminus V(G')$, there is a subset $A$ of $S$ such that $N_G[v] = N_G[A]$. By Corollary~\ref{coro: equal CN}, the closed neighborhood of every vertex in $G$ can be determined. Therefore, $G$ is reconstructible from $\mathscr{D}(G)$.
\end{proof}

\section*{Concluding remarks}

In this paper, we have shown that every $C_4$-free graph is strongly reconstructible from either of the support $\supp(\mathcal{N}[G])$ of its closed neighborhoods or its digital convexity $\mathscr{D}(G)$. Indeed, the class of $C_4$-free graphs is the maximal induced-subgraph-closed class in which every graph is reconstructible from its set of closed neighborhoods. To see this, first we observe that every induced-subgraph-closed class that is not a subclass of the class of $C_4$-free graphs must contain $C_4$ as a member. However, a labeled $C_4$ is not strongly reconstructible from its set of closed neighborhoods. Figure~\ref{fig_C4} shows all three distinct labelings of $C_4$, each having the same set of closed neighborhoods.

\begin{figure}[htb]
\hbox to \hsize{
\hfil
\resizebox{8cm}{!}{\tikzset{every picture/.style={line width=0.75pt}} 

\begin{tikzpicture}[x=0.75pt,y=0.75pt,yscale=-1,xscale=1]

\draw  [fill={rgb, 255:red, 0; green, 0; blue, 0 }  ,fill opacity=1 ] (104.55,140) .. controls (104.55,138.1) and (106.1,136.55) .. (108,136.55) .. controls (109.9,136.55) and (111.45,138.1) .. (111.45,140) .. controls (111.45,141.9) and (109.9,143.45) .. (108,143.45) .. controls (106.1,143.45) and (104.55,141.9) .. (104.55,140) -- cycle ;
\draw   (108,90) -- (158,90) -- (158,140) -- (108,140) -- cycle ;
\draw  [fill={rgb, 255:red, 0; green, 0; blue, 0 }  ,fill opacity=1 ] (104.55,90) .. controls (104.55,88.1) and (106.1,86.55) .. (108,86.55) .. controls (109.9,86.55) and (111.45,88.1) .. (111.45,90) .. controls (111.45,91.9) and (109.9,93.45) .. (108,93.45) .. controls (106.1,93.45) and (104.55,91.9) .. (104.55,90) -- cycle ;
\draw  [fill={rgb, 255:red, 0; green, 0; blue, 0 }  ,fill opacity=1 ] (154.55,90) .. controls (154.55,88.1) and (156.1,86.55) .. (158,86.55) .. controls (159.9,86.55) and (161.45,88.1) .. (161.45,90) .. controls (161.45,91.9) and (159.9,93.45) .. (158,93.45) .. controls (156.1,93.45) and (154.55,91.9) .. (154.55,90) -- cycle ;
\draw  [fill={rgb, 255:red, 0; green, 0; blue, 0 }  ,fill opacity=1 ] (154.55,140) .. controls (154.55,138.1) and (156.1,136.55) .. (158,136.55) .. controls (159.9,136.55) and (161.45,138.1) .. (161.45,140) .. controls (161.45,141.9) and (159.9,143.45) .. (158,143.45) .. controls (156.1,143.45) and (154.55,141.9) .. (154.55,140) -- cycle ;

\draw  [fill={rgb, 255:red, 0; green, 0; blue, 0 }  ,fill opacity=1 ] (222.05,140) .. controls (222.05,138.1) and (223.6,136.55) .. (225.5,136.55) .. controls (227.4,136.55) and (228.95,138.1) .. (228.95,140) .. controls (228.95,141.9) and (227.4,143.45) .. (225.5,143.45) .. controls (223.6,143.45) and (222.05,141.9) .. (222.05,140) -- cycle ;
\draw   (225.5,90) -- (275.5,90) -- (275.5,140) -- (225.5,140) -- cycle ;
\draw  [fill={rgb, 255:red, 0; green, 0; blue, 0 }  ,fill opacity=1 ] (222.05,90) .. controls (222.05,88.1) and (223.6,86.55) .. (225.5,86.55) .. controls (227.4,86.55) and (228.95,88.1) .. (228.95,90) .. controls (228.95,91.9) and (227.4,93.45) .. (225.5,93.45) .. controls (223.6,93.45) and (222.05,91.9) .. (222.05,90) -- cycle ;
\draw  [fill={rgb, 255:red, 0; green, 0; blue, 0 }  ,fill opacity=1 ] (272.05,90) .. controls (272.05,88.1) and (273.6,86.55) .. (275.5,86.55) .. controls (277.4,86.55) and (278.95,88.1) .. (278.95,90) .. controls (278.95,91.9) and (277.4,93.45) .. (275.5,93.45) .. controls (273.6,93.45) and (272.05,91.9) .. (272.05,90) -- cycle ;
\draw  [fill={rgb, 255:red, 0; green, 0; blue, 0 }  ,fill opacity=1 ] (272.05,140) .. controls (272.05,138.1) and (273.6,136.55) .. (275.5,136.55) .. controls (277.4,136.55) and (278.95,138.1) .. (278.95,140) .. controls (278.95,141.9) and (277.4,143.45) .. (275.5,143.45) .. controls (273.6,143.45) and (272.05,141.9) .. (272.05,140) -- cycle ;

\draw  [fill={rgb, 255:red, 0; green, 0; blue, 0 }  ,fill opacity=1 ] (339.55,140) .. controls (339.55,138.1) and (341.1,136.55) .. (343,136.55) .. controls (344.9,136.55) and (346.45,138.1) .. (346.45,140) .. controls (346.45,141.9) and (344.9,143.45) .. (343,143.45) .. controls (341.1,143.45) and (339.55,141.9) .. (339.55,140) -- cycle ;
\draw   (343,90) -- (393,90) -- (393,140) -- (343,140) -- cycle ;
\draw  [fill={rgb, 255:red, 0; green, 0; blue, 0 }  ,fill opacity=1 ] (339.55,90) .. controls (339.55,88.1) and (341.1,86.55) .. (343,86.55) .. controls (344.9,86.55) and (346.45,88.1) .. (346.45,90) .. controls (346.45,91.9) and (344.9,93.45) .. (343,93.45) .. controls (341.1,93.45) and (339.55,91.9) .. (339.55,90) -- cycle ;
\draw  [fill={rgb, 255:red, 0; green, 0; blue, 0 }  ,fill opacity=1 ] (389.55,90) .. controls (389.55,88.1) and (391.1,86.55) .. (393,86.55) .. controls (394.9,86.55) and (396.45,88.1) .. (396.45,90) .. controls (396.45,91.9) and (394.9,93.45) .. (393,93.45) .. controls (391.1,93.45) and (389.55,91.9) .. (389.55,90) -- cycle ;
\draw  [fill={rgb, 255:red, 0; green, 0; blue, 0 }  ,fill opacity=1 ] (389.55,140) .. controls (389.55,138.1) and (391.1,136.55) .. (393,136.55) .. controls (394.9,136.55) and (396.45,138.1) .. (396.45,140) .. controls (396.45,141.9) and (394.9,143.45) .. (393,143.45) .. controls (391.1,143.45) and (389.55,141.9) .. (389.55,140) -- cycle ;

\draw (90.23,78.75) node [anchor=north west][inner sep=0.75pt]    {$1$};
\draw (166.23,78.75) node [anchor=north west][inner sep=0.75pt]    {$2$};
\draw (90.23,134.75) node [anchor=north west][inner sep=0.75pt]    {$4$};
\draw (166.23,134.75) node [anchor=north west][inner sep=0.75pt]    {$3$};
\draw (207.73,78.75) node [anchor=north west][inner sep=0.75pt]    {$1$};
\draw (283.73,78.75) node [anchor=north west][inner sep=0.75pt]    {$2$};
\draw (207.73,134.75) node [anchor=north west][inner sep=0.75pt]    {$3$};
\draw (283.73,134.75) node [anchor=north west][inner sep=0.75pt]    {$4$};
\draw (325.23,78.75) node [anchor=north west][inner sep=0.75pt]    {$1$};
\draw (401.23,78.75) node [anchor=north west][inner sep=0.75pt]    {$3$};
\draw (325.23,134.75) node [anchor=north west][inner sep=0.75pt]    {$4$};
\draw (401.23,134.75) node [anchor=north west][inner sep=0.75pt]    {$2$};

\end{tikzpicture}}%
\hfil
}
\caption{$C_4$ with three different labelings.}
\label{fig_C4}
\end{figure}
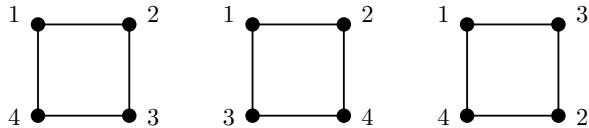

On the other hand, the $C_4$-free condition is not necessary for such reconstructions. For instance, it is straightforward to verify that the graph in Figure~\ref{fig_contains_C4} is the unique graph whose set of closed neighborhoods is
\[
\big\{\{1,2,4,5\},\ \{1,2,3\},\ \{2,3,4\},\ \{1,3,4\},\ \{1,5\}\big\},
\]
and that $G$ contains an induced $C_4$. Therefore, providing a precise characterization of graphs that are reconstructible from the support of their closed neighborhoods remains an open problem.
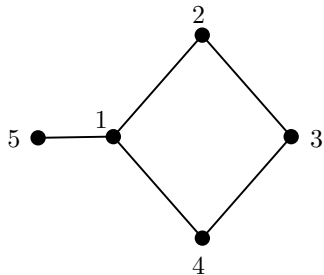
\begin{figure}[htb]
\hbox to \hsize{
\hfil
\resizebox{4.5cm}{!}{\tikzset{every picture/.style={line width=0.75pt}} 

\begin{tikzpicture}[x=0.75pt,y=0.75pt,yscale=-1,xscale=1]

\draw  [fill={rgb, 255:red, 0; green, 0; blue, 0 }  ,fill opacity=1 ] (205.4,90.6) .. controls (205.4,88.7) and (206.95,87.15) .. (208.85,87.15) .. controls (210.75,87.15) and (212.3,88.7) .. (212.3,90.6) .. controls (212.3,92.5) and (210.75,94.05) .. (208.85,94.05) .. controls (206.95,94.05) and (205.4,92.5) .. (205.4,90.6) -- cycle ;
\draw  [fill={rgb, 255:red, 0; green, 0; blue, 0 }  ,fill opacity=1 ] (159.55,143) .. controls (159.55,141.1) and (161.1,139.55) .. (163,139.55) .. controls (164.9,139.55) and (166.45,141.1) .. (166.45,143) .. controls (166.45,144.9) and (164.9,146.45) .. (163,146.45) .. controls (161.1,146.45) and (159.55,144.9) .. (159.55,143) -- cycle ;
\draw  [fill={rgb, 255:red, 0; green, 0; blue, 0 }  ,fill opacity=1 ] (251.25,143) .. controls (251.25,141.1) and (252.8,139.55) .. (254.7,139.55) .. controls (256.6,139.55) and (258.15,141.1) .. (258.15,143) .. controls (258.15,144.9) and (256.6,146.45) .. (254.7,146.45) .. controls (252.8,146.45) and (251.25,144.9) .. (251.25,143) -- cycle ;
\draw  [fill={rgb, 255:red, 0; green, 0; blue, 0 }  ,fill opacity=1 ] (205.46,195.35) .. controls (205.46,193.45) and (207,191.91) .. (208.9,191.91) .. controls (210.81,191.91) and (212.35,193.45) .. (212.35,195.35) .. controls (212.35,197.26) and (210.81,198.8) .. (208.9,198.8) .. controls (207,198.8) and (205.46,197.26) .. (205.46,195.35) -- cycle ;
\draw    (163,143) -- (208.9,195.35) ;
\draw    (208.8,90.65) -- (254.7,143) ;
\draw    (254.7,143) -- (208.9,195.35) ;
\draw    (163,143) -- (208.85,90.6) ;
\draw    (163.05,142.95) -- (124.2,143.6) ;
\draw  [fill={rgb, 255:red, 0; green, 0; blue, 0 }  ,fill opacity=1 ] (120.75,143.6) .. controls (120.75,141.7) and (122.3,140.15) .. (124.2,140.15) .. controls (126.1,140.15) and (127.65,141.7) .. (127.65,143.6) .. controls (127.65,145.5) and (126.1,147.05) .. (124.2,147.05) .. controls (122.3,147.05) and (120.75,145.5) .. (120.75,143.6) -- cycle ;

\draw (152,128.4) node [anchor=north west][inner sep=0.75pt]    {$1$};
\draw (202.23,74.4) node [anchor=north west][inner sep=0.75pt]    {$2$};
\draw (263,138.4) node [anchor=north west][inner sep=0.75pt]    {$3$};
\draw (202.23,203.75) node [anchor=north west][inner sep=0.75pt]    {$4$};
\draw (107,138.4) node [anchor=north west][inner sep=0.75pt]    {$5$};

\end{tikzpicture}}%
\hfil
}
\caption{$G$ contains an induced $C_4$.}
\label{fig_contains_C4}
\end{figure}

\subsection*{Acknowledgments.} This project was initiated at the Graduate Research Workshop in Combinatorics (GRWC) held at the University of Wisconsin-Milwaukee in 2024, which was supported by the Combinatorics Foundation and the National Science Foundation (NSF Grant DMS-1953445). Borgwardt's work was supported by the Air Force Office of Scientific Research under award number FA9550-24-1-0240. We thank Jacob Dunham, Alice Lacaze-Masmonteil and Chelsea Sato for helpful discussions in the early stage of this project. We also thank Bryce Frederickson for helping us find the first pair of non-isomorphic graphs that share the same sets of digital convexity.
\bibliographystyle{abbrv}
\bibliography{References}

\begin{thebibliography}{10}

\bibitem{AignerTriesch1993}
M.~Aigner and E.~Triesch.
\newblock Reconstructing a graph from its neighborhood lists.
\newblock {\em Combinatorics, Probability and Computing}, 2(2):103--113, 1993.

\bibitem{Bondy}
J.~A. Bondy.
\newblock On {U}lam's conjecture for separable graphs.
\newblock {\em Pacific Journal of Mathematics}, 31:281--288, 1969.

\bibitem{Bondy-Murty}
J.~A. Bondy and U.~S.~R. Murty.
\newblock {\em Graph theory}, volume 244 of {\em Graduate Texts in Mathematics}.
\newblock Springer, 2008.

\bibitem{Fomin}
F.~V. Fomin, J.~Kratochv\'il, D.~Lokshtanov, F.~Mancini, and J.~A. Telle.
\newblock On the complexity of reconstructing {$H$}-free graphs from their star systems.
\newblock {\em Journal of Graph Theory}, 68(2):113--124, 2011.

\bibitem{Hammack-Mullican}
R.~H. Hammack and C.~Mullican.
\newblock Neighborhood reconstruction and cancellation of graphs.
\newblock {\em Electronic Journal of Combinatorics}, 24(2):Paper No. 2.8, 11, 2017.

\bibitem{Harary}
F.~Harary.
\newblock On the reconstruction of a graph from a collection of subgraphs.
\newblock In {\em Theory of {G}raphs and its {A}pplications (Proceedings of the Symposium Held in Smolenice, 1963)}, pages 47--52. Publ. House Czech. Acad. Sci., Prague, 1964.

\bibitem{HararySurvey}
F.~Harary.
\newblock A survey of the reconstruction conjecture.
\newblock In {\em Graphs and combinatorics ({P}roceedings of the {C}apital {C}onference, {G}eorge {W}ashington {U}niversity, {W}ashington, {D}.{C}., 1973)}, volume Vol. 406 of {\em Lecture Notes in Math.}, pages 18--28. Springer, Berlin-New York, 1974.

\bibitem{Kelly}
P.~J. Kelly.
\newblock A congruence theorem for trees.
\newblock {\em Pacific Journal of Mathematics}, 7:961--968, 1957.

\bibitem{LAFRANCE2017254}
P.~Lafrance, O.~R. Oellermann, and T.~Pressey.
\newblock Reconstructing trees from digitally convex sets.
\newblock {\em Discrete Applied Mathematics}, 216:254--260, 2017.

\bibitem{Ulam}
S.~M. Ulam.
\newblock {\em A collection of mathematical problems}, volume no. 8 of {\em Interscience Tracts in Pure and Applied Mathematics}.
\newblock Interscience Publishers, New York-London, 1960.

\bibitem{vonRimscha}
M.~von Rimscha.
\newblock Reconstructibility and perfect graphs.
\newblock {\em Discrete Mathematics}, 47(2-3):283--291, 1983.

\end{thebibliography}

\end{document}